\documentclass{amsart}
\title{Kronecker Comultiplication of Stable Characters and Restriction From $S_{mn}$ to $S_m \times S_n$}
\author{Christopher Ryba}
\address{Department of Mathematics, University of California, Berkeley, CA 94720, USA}
\email{ryba@math.berkeley.edu}
\usepackage{hyperref}
\usepackage{fullpage}
\usepackage{amsmath}
\usepackage{amsfonts}
\usepackage{amssymb}
\usepackage{ytableau}
\usepackage{subfig}
\usepackage{float}
\usepackage{graphicx, caption}

\DeclareMathOperator{\Id}{Id}

\newcommand{\Res}[2] {
  \textnormal{Res}_{#1}^{#2} 
}

\begin{document}
\maketitle
\begin{abstract}
A family of symmetric functions $\tilde{s}_\lambda$ was introduced in \cite{OrellanaZabrocki}, and independently in \cite{AssafSpeyer}. The $\tilde{s}_\lambda$ encode many stability properties of representations of symmetric groups (e.g. when multiplied, the structure constants are reduced Kronecker coefficients). We show that the structure constants for the Kronecker comultiplication $\Delta^*$ are multiplicities for the restriction of irreducible representations from $S_{mn}$ to $S_m \times S_n$ (provided $m$ and $n$ are sufficiently large), and use the structure of $\tilde{s}_\lambda$ to demonstrate two-row stability properties of these restriction multiplicities.

\end{abstract}
\newtheorem{theorem}{Theorem}[section]
\newtheorem{lemma}[theorem]{Lemma}
\newtheorem{proposition}[theorem]{Proposition}
\newtheorem{corollary}[theorem]{Corollary}
\newtheorem{definition}[theorem]{Definition}
\newtheorem{example}[theorem]{Example}
\newtheorem{remark}[theorem]{Remark}
\newtheorem{observation}[theorem]{Observation}

\section{Introduction}
\noindent
In this paper, we view elements of the symmetric group $S_n$ as permutation matrices. This realises $S_n$ as a subgroup of the general linear group $GL_n(\mathbb{C})$. We write $S^\mu$ for the irreducible complex representation of $S_n$ labelled by a partition $\mu$ of size $n$. We also write $\mathbb{S}^\lambda$ for the Schur functor indexed by the partition $\lambda$, so that the $\mathbb{S}^\lambda(\mathbb{C}^n)$ with $\lambda$ having at most $n$ rows are the irreducible polynomial representations of $GL_n(\mathbb{C})$. We may restrict $\mathbb{S}^\lambda(\mathbb{C}^n)$ to a representation of $S_n$. The restriction multiplicities are given by the following formula \cite{Gay} (see also Exercise 7.74 of \cite{Stanley}). Let $\mu$ be a partition of $n$, then
\begin{equation} \label{restriction_multiplicity_equation}
\dim \hom_{S_n} (S^\mu, \mathbb{S}^\lambda(\mathbb{C}^n)) = \langle s_\lambda, s_\mu[1 + h_1 + h_2 + \cdots] \rangle,
\end{equation}
where $s_\lambda, s_\mu$ are Schur functions, $h_r$ are the complete symmetric functions, $\langle -,- \rangle$ is the Hall inner product, and square brackets indicate plethysm. However, it is an open problem to find combinatorial formula for this restriction multiplicity.
\newline \newline \noindent
One approach to studying restriction from $GL_n(\mathbb{C})$ to $S_n$ is via the symmetric functions $\tilde{s}_\mu$ introduced in \cite{OrellanaZabrocki} and independently in \cite{AssafSpeyer}. The $\tilde{s}_\mu$ form a basis of the ring of symmetric functions, $\Lambda$. If $\mu[n] = (n-|\mu|, \mu)$ is the partition of $n$ obtained by adding a row of the appropriate length at the start of $\mu$ (defined for $n$ large enough), then the character of $S^{\mu[n]}$ on an element $g \in S_n$ is given by an evaluation of $\tilde{s}_\mu$:
\[
\chi^{\mu[n]}(g) = \tilde{s}_\mu(z_i),
\]
where $z_i$ are the eigenvalues of $g$ when viewed as a permutation matrix. Because the character of $\mathbb{S}^{\lambda}(\mathbb{C}^n)$ at $g$ is obtained by evaluating the Schur function $s_\lambda$ at the eigenvalues of $g$, it follows that the entries of the change-of-basis matrix between the Schur functions $s_\lambda$ and the $\tilde{s}_\mu$ are the multiplicities of $S^{\mu[n]}$ in the restriction of $\mathbb{S}^{\lambda}(\mathbb{C}^n)$ to $S_n$, provided $n$ is large enough. In particular, the multiplicity of $S^{\mu[n]}$ in $\mathbb{S}^\lambda(\mathbb{C}^n)$ is eventually constant as $n$ grows; this is one example of a \emph{stability phenomenon} in the representation theory of symmetric groups. The topic of stability phenomena is broad and deep; \cite{SamSnowden3} is an excellent reference.
\newline \newline \noindent
The Kronecker coefficient $k_{\mu, \nu}^\lambda$ is the multiplicity of $S^{\lambda}$ in $S^{\mu} \otimes S^{\nu}$. (It is a century-old open problem to give a combinatorial interpretation of the Kronecker coefficients.) Let $\tilde{k}_{\mu,\nu}^\lambda$ be the multiplicative structure constants of the $\tilde{s}_\lambda$:
\[
\tilde{s}_\mu \tilde{s}_\nu = \sum_{\lambda}\tilde{k}_{\mu, \nu}^\lambda \tilde{s}_\lambda.
\]
Evaluating at the eigenvalues of $g \in S_n$ (where $n$ is sufficiently large), we obtain
\[
\chi^{\mu[n]}(g) \chi^{\nu[n]}(g) = \sum_{\lambda} \tilde{k}_{\mu, \nu}^{\lambda} \chi^{\lambda[n]}(g),
\]
showing that $\tilde{k}_{\mu, \nu}^{\lambda} = k_{\mu[n], \nu[n]}^{\lambda[n]}$ for $n$ sufficiently large (see Theorem 7 of \cite{OrellanaZabrocki}). The fact that these tensor product multiplicities are eventually constant as $n$ grows is called \emph{stability of Kronecker coefficients} and is due to Murnaghan \cite{Murnaghan}. The stable limits $\tilde{k}_{\mu, \nu}^{\lambda}$ are called reduced Kronecker coefficients.
\newline \newline \noindent
The behaviour of $\tilde{s}_\lambda$ under the ``usual'' comultiplication $\Delta$ on $\Lambda$ (for which power-sum symmetric functions are primitive) was studied in \cite{OrellanaZabrocki2}, where it was shown that the structure constants correspond to stable versions of Littlewood-Richardson coefficients, encoding multiplicities for the restriction from $S_{m+n}$ to $S_m \times S_n$ for $m$ and $n$ both large. In the present paper, we instead study the behaviour of $\tilde{s}_\lambda$ under the Kronecker comultiplication $\Delta^*$ (for which power-sum symmetric functions are grouplike) defining structure constants $R_{\mu, \nu}^\lambda$ via
\[
\Delta^*(\tilde{s}_\lambda) = \sum_{\mu, \nu} R_{\mu, \nu}^{\lambda} \tilde{s}_{\mu} \otimes \tilde{s}_\nu.
\]
We find that these may be interpreted in the following way. If we write $[k] = \{1,2,\ldots,k\}$, then $[m] \times [n]$ is a set with $mn$ elements. There is an action of $S_m \times S_n$ where each symmetric group acts on the corresponding factor, and we obtain an embedding of $S_m \times S_n$ into $S_{mn}$. It turns out that $R_{\mu, \nu}^{\lambda}$ is the multiplicity of $S^{\mu[m]} \otimes S^{\nu[n]}$ in the restriction
\[
\Res{S_m \times S_n}{S_{mn}}(S^{\lambda[mn]}),
\]
provided that both $m$ and $n$ are large enough. The restriction from $S_{mn}$ to $S_m \times S_n$ appears in the theory of combinatorial species; see Exercise 4.4.9 (and following discussion) in \cite{GrinbergReiner} for an explanation of the connection to the arithmetic product of combinatorial species defined in \cite{MaiaMendez}.
\newline \newline \noindent
In Theorem \ref{R_stability_theorem}, we show that the $R_{\mu, \nu}^{\lambda}$ themselves exhibit a stability phenomenon: for any integers $a,b$ and partitions $\mu, \nu, \lambda$,
\[
R_{\mu[n-a], \nu[n-b]}^{\lambda[n]}
\]
is eventually constant as $n$ grows. Curiously, this implies a two-row stability property for the restriction multiplicities from $S_{pq}$ to $S_p \times S_q$: for any $a,b$ we have that
\[
\lim_{n \to \infty} \lim_{p,q \to \infty} \dim\hom_{S_p \times S_q}\left(S^{\mu[n-a][p]} \otimes S^{\nu[n-b][q]}, \Res{S_p \times S_q}{S_{pq}} \left( S^{\lambda[n][pq]} \right) \right)
\]
exists and is finite, where $\rho[r][s] = (r - |\rho|, \rho)[s] = (s-r, r-|\rho|, \rho)$.
\newline \newline \noindent
The paper is organised as follows. In Section 2, we review the properties of symmetric functions that we will require. Then we discuss the symmetric functions $\tilde{s}_\lambda$ in Section 3. In Section 4, we explain the interpretation of $R_{\mu,\nu}^{\lambda}$ in terms of restriction from $S_{mn}$ to $S_m \times S_n$. Finally, we prove the two-row stability property in Section 5.

\section*{Acknowledgements}
\noindent
The author would like to thank Pavel Etingof for Remark \ref{deligne_category_remark}, as well as Nate Harman, Darij Grinberg, and Victor Reiner for useful conversations.

\section{Symmetric Functions}
\noindent
The ring of symmetric functions, $\Lambda$, may be viewed as the ring of ``polynomials that are symmetric in infinitely many variables $x_1, x_2, \ldots$''. We direct the reader to Chapter 1 of \cite{Macdonald} for the precise construction of $\Lambda$ and as a general reference for this section.
\newline \newline \noindent
Recall that a partition is a finite, weakly-decreasing sequence of positive integers $\lambda = (\lambda_1, \lambda_2, \ldots, \lambda_m)$, where the $\lambda_i$ are called the \emph{parts} of $\lambda$. We write $|\lambda| = \sum_i \lambda_i$ for the \emph{size} of $\lambda$, and $l(\lambda) = m$ for the \emph{length} of $\lambda$. An alternative notation for partitions is to write $\lambda = (1^{m_1} 2^{m_2} \cdots)$, where $m_i$ is the number of times $i$ appears as a part of $\lambda$. We write $m_i(\lambda)$ if it us unclear which partition we are referring to. The \emph{dual} (or \emph{transpose}) partition to $\lambda$ is $\lambda^\prime$, which is defined by $\lambda_i^{\prime} = |\{j \mid \lambda_j \geq i\}|$. We have $|\lambda^\prime| = |\lambda|$ and $\lambda^{\prime \prime} = \lambda$. We will depict partitions with Young diagrams in English notation. For example, the partition $(5,4,1)$ is drawn as follows:

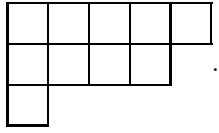
\begin{figure}[H]
\ytableausetup{centertableaux}
\ydiagram{5,4,1}.
\end{figure}
\noindent
As an algebra, $\Lambda = \mathbb{Z}[h_1, h_2, \ldots]$ is the free polynomial algebra generated by the complete symmetric functions $h_r$. Each $h_r$ has degree $r$, making $\Lambda$ into a graded algebra. Passing to rational coefficients, we have
\[
\mathbb{Q} \otimes_{\mathbb{Z}} \Lambda = \mathbb{Q}[p_1, p_2, \ldots],
\]
so that over $\mathbb{Q}$, $\Lambda$ is also freely generated by the power-sum symmetric functions $p_r = \sum_i x_i^r$ (having degree $r$). Hence, over $\mathbb{Q}$, $\Lambda$ has a basis
\[
p_\lambda = p_{\lambda_1} p_{\lambda_2} \cdots p_{\lambda_m}
\]
indexed by all partitions $\lambda$. The relation between the complete and power-sum symmetric functions is encoded in the following equality of generating functions in the variable $t$:
\[
\sum_{i \geq 0} h_i t^i = \prod_i \frac{1}{1-x_i t} = \exp\left( \sum_{i \geq 1} \frac{p_i}{i} t^i \right).
\]
We will actually want to evaluate at $t=1$, thereby obtaining an element of the completion of $\Lambda$ with respect to the grading; let
\[
H = \sum_{i \geq 0} h_i = \prod_i \frac{1}{1-x_i} = \exp\left( \sum_{i \geq 1} \frac{p_i}{i} \right).
\]
\newline \newline \noindent
However, $\Lambda$ also has $\mathbb{Z}$-basis given by the Schur functions $s_\lambda$ (which are also indexed by all partitions $\lambda$). A very important property of Schur functions is that they are characters of representations of the general linear groups. For $g \in GL_n(\mathbb{C})$, the trace of $g$ on $\mathbb{S}^\lambda(\mathbb{C}^n)$ is given by
\[
tr_{\mathbb{S}^{\lambda}(\mathbb{C}^n)}(g) = s_{\lambda}(z_1, z_2, \ldots, z_n),
\]
where $s_{\lambda}(z_1, z_2, \ldots, z_n)$ is the evaluation of the Schur function $s_{\lambda}$ at the eigenvalues $z_i$ of $g$ viewed as an $n \times n$ matrix.
\newline \newline \noindent
There is a symmetric bilinear form $\langle -,- \rangle$ on $\Lambda$ called the Hall inner product. It can be defined as follows: if $\mu, \nu$ are partitions, then
\[
\langle p_\mu, p_\nu \rangle = \delta_{\mu, \nu} z_{\mu},
\]
where if $\mu = (1^{m_1} 2^{m_2} \cdots)$, then $z_\mu = \prod_i m_i! i^{m_i}$, and $\delta_{\mu,\nu}$ is the Kronecker delta. The Schur functions are orthonormal for this inner product, so for any $f \in \Lambda$ we have the identity
\[
f = \sum_\lambda \langle f, s_\lambda \rangle s_\lambda.
\]
We write $f^\perp$ for the adjoint (with respect to $\langle -,- \rangle$) of multiplication by $f$. It turns out that 
\[
p_r^\perp = r \frac{\partial}{\partial p_r},
\]
where we view $\mathbb{Q} \otimes \Lambda = \mathbb{Q}[p_1, p_2, \ldots]$ (see \cite{Macdonald} Section 1.5, Example 3(c)). Since $p_r$ is homogeneous of degree $r$, $p_r^\perp$ is homogeneous of degree $-r$. We will use the following notation: if $\beta = (\beta_1, \beta_2, \ldots, \beta_m)$, then
\[
\frac{\partial}{\partial p_\beta} =  \frac{\partial}{\partial p_{\beta_1}} \frac{\partial}{\partial p_{\beta_2}} \cdots \frac{\partial}{\partial p_{\beta_m}},
\]
which is an operator of degree $-|\beta|$. We refer to the subalgebra of $\mathrm{End}_{\mathbb{Q}}(\mathbb{Q} \otimes \Lambda)$ generated by multiplication by symmetric functions $f$ as well as the adjoints $f^\perp$ as \emph{differential operators}. Expressing $f$ as a linear combination of the $p_\alpha$, it follows that any differential operator is a $\mathbb{Q}$-linear combination of $p_\alpha \frac{\partial}{\partial p_\beta}$ for partitions $\alpha, \beta$.
\newline \newline \noindent
Suppose $\mu$ and $\lambda$ are partitions such that the Young diagram of $\mu$ contains the Young diagram of $\lambda$. Recall that the complement $\mu-\lambda$ is called a \emph{rim hook} if it consists of a single connected component (vertical and horizontal boxes are adjacent, but diagonal ones are not) and contains no $2 \times 2$ square. The \emph{height} $ht(\mu-\lambda)$ is one less than the number of rows which $\mu-\lambda$ intersects. 
\begin{example}
Here are several examples of rim hooks (indicated in grey). Their heights are 0, 1, and 2 respectively.
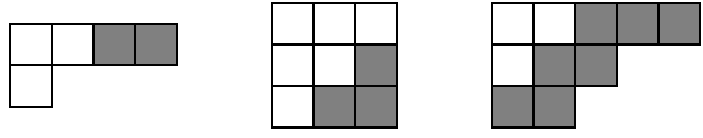
\begin{figure}[H]
\centering
\ytableausetup{nosmalltableaux}
\begin{ytableau}
*(white) &*(white) &*(gray) &*(gray) \\
*(white)
\end{ytableau}
\hspace{10mm}
\begin{ytableau}
*(white) &*(white) &*(white) \\
*(white) & *(white) &*(gray) \\
*(white) &*(gray) &*(gray)
\end{ytableau}
\hspace{10mm}
\begin{ytableau}
*(white) &*(white) &*(gray) &*(gray) & *(gray) \\
*(white) & *(gray) & *(gray) \\
*(gray) &*(gray)
\end{ytableau}
\end{figure}
\end{example}
\noindent
The action of $p_r^\perp$ on Schur functions may be described combinatorially in terms of rim-hooks (\cite{Macdonald} Section 1.3, Example 11). We have 
\[
p_r s_{\lambda} = \sum_{\mu} (-1)^{ht(\mu-\lambda)} s_\mu
\]
where the sum is over all partitions $\mu$ such that the diagram $\mu-\lambda$ is a rim hook of size $r$. From this it follows that
\[
p_r^\perp s_{\mu} = \sum_{\lambda} (-1)^{ht(\mu - \lambda)} s_\lambda,
\]
where the sum is now over partitions $\lambda$ whose Young diagrams are obtained from the Young diagram of $\mu$ by removing a rim hook of size $r$.
\begin{observation} \label{rim_hook_observation}
Suppose that we wish to remove a rim hook of size $r$ from $\mu[n] = (n-|\mu|, \mu)$, where $n \geq 2|\mu| + r$. If this rim-hook intersects the top row, it must contain the final box in the top row, and hence must be entirely contained in the top row (since the top row has at least $|\mu|+r$ boxes, and the second row has at most $|\mu|$ boxes, $r$ boxes do not suffice to reach the second row) and the result is $\mu[n-r]$. If the rim hook does not intersect the top row, then the result is $\lambda[n-r]$, where $\lambda$ is obtained from $\mu$ by removing a rim hook of size $r$. For example, if $\mu = (2,2)$ here are the rim hooks of size 2 in the partition $\mu[10]$:
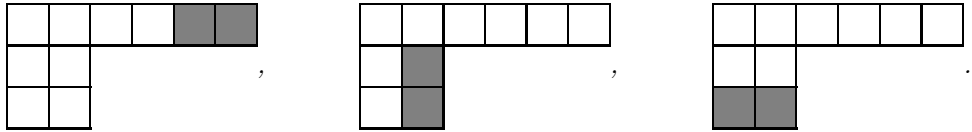
\begin{figure}[H]
\centering
\ytableausetup{nosmalltableaux}
\begin{ytableau}
*(white) &*(white) &*(white) &*(white) &*(gray) &*(gray) \\
*(white) &*(white) \\
*(white) &*(white)
\end{ytableau},
\hspace{10mm}
\begin{ytableau}
*(white) &*(white) &*(white) &*(white) &*(white) &*(white) \\
*(white) &*(gray) \\
*(white) &*(gray)
\end{ytableau},
\hspace{10mm}
\begin{ytableau}
*(white) &*(white) &*(white) &*(white) &*(white) &*(white) \\
*(white) &*(white) \\
*(gray) &*(gray)
\end{ytableau}.
\end{figure}
\noindent
The first rim hook is contained within the first row, and the remainder are contained in $\mu$. Note in particular that removing a rim hook of size $2$ from $\mu[n]$ for $n \geq 10$ would give identical results except the top rows would each have $n-10$ extra boxes.
\end{observation}
\noindent
The usual Hopf algebra structure on $\Lambda$ is as follows (\cite{Macdonald} Section 1.5 Example 25). The comultiplication $\Delta$ satisfies $\Delta(p_r) = p_r \otimes 1 + 1 \otimes p_r$. The counit $\varepsilon$ behaves as follows: $\varepsilon(p_\mu)$ is 1 if $\mu$ is the empty partition and zero otherwise. The antipode $S$ is a homomorphism satisfying $S(p_\mu) = (-1)^{l(\mu)}p_\mu$, and in terms of the Schur basis, $S(s_\lambda) = (-1)^{|\lambda|} s_{\lambda^\prime}$. In particular, this implies $S$ is an isometry:
\[
\langle S(f), S(g) \rangle = \langle f, g \rangle,
\]
and $S$ is an involution: $S^2 = \Id$. 
\begin{lemma} \label{antipode_adjoint_lemma}
We have $(-1)^{|\lambda|}S(p_\mu s_{\lambda^\prime}) = (-1)^{l(\mu)} p_\mu s_\lambda$.
\end{lemma}
\begin{proof}
It easily follows that
\begin{eqnarray*}
\langle (-1)^{|\lambda|} S(p_\mu^\perp s_{\lambda^\prime}), f \rangle &=& (-1)^{|\lambda|} \langle s_{\lambda^\prime}, p_\mu S(f) \rangle \\
&=& (-1)^{|\lambda|} \langle S(s_{\lambda^\prime}), S(p_\mu) f \rangle \\
&=& \langle s_\lambda, (-1)^{l(\mu)} p_\mu f \rangle \\
&=& \langle (-1)^{l(\mu)} p_\mu^\perp s_\lambda, f \rangle,
\end{eqnarray*}
and the lemma follows from the non-degeneracy of the Hall inner product.
\end{proof}
\noindent
The comultiplication may be interpreted in the following way. An element of $\Lambda \otimes \Lambda$ may be viewed as a function that is symmetric in two sets of variables $x_i$ and $y_j$ separately (one for each tensor factor of $\Lambda$). So we may obtain an element of $\Lambda \otimes \Lambda$ by evaluating an element of $\Lambda$ at the variable set $\{x_i\} \cup \{y_j\}$. It turns out that the Hall inner product is a \emph{Hopf pairing}, meaning that if we equip $\Lambda \otimes \Lambda$ with the form 
\[
\langle a \otimes b, c \otimes d\rangle  = \langle a,c \rangle \langle b,d \rangle,
\]
then we have
\[
\langle f, gh \rangle = \langle \Delta(f), g \otimes h \rangle.
\]
At one point we will want to compute $\Delta(f)$ in the following way. 
\begin{lemma}[\cite{Macdonald}, Section 1.5, Example 25(b)] \label{standard_coproduct_formula_lemma}
For any $f \in \Lambda$ we have
\[
\Delta(f) = \sum_\mu p_\mu^\perp f \otimes \frac{p_\mu}{z_\mu}.
\]
\end{lemma}
\noindent
We will need to consider a different comultiplication, $\Delta^*$, obeying
\[
\Delta^*(p_\mu) = p_\mu \otimes p_\mu,
\]
which together with the counit, $\varepsilon^*$, satisfying $\varepsilon^*(p_\mu) = 1$, makes $\Lambda$ into a bialgebra (\cite{Macdonald} Section 1.7, Example 20). Analogously to $\Delta$, the comultiplication $\Delta^*$ may be described as evaluating at the variable set $\{x_i y_j\}$ (consisting of all products of the variables $x_i$ and $y_j$) to obtain an element of $\Lambda \otimes \Lambda$. In terms of Schur functions, we have
\[
\Delta^*(s_\lambda) = \sum_{|\mu| = |\nu| = |\lambda|} k_{\mu, \nu}^\lambda s_\mu \otimes s_\nu,
\]
where $k_{\mu, \nu}^\lambda$ are Kronecker coefficients. Hence, we call $\Delta^*$ the \emph{Kronecker comultiplication}.
\newline \newline \noindent
Finally in our review of symmetric functions, we will need \emph{plethysm}. Given two symmetric functions $f$ and $g$, the plethysm of $f[g]$ is the symmetric function obtained by evaluating $f$ at the variable set $\{z_i\}$, where $z_i$ are the monomials appearing in $g$ counted with multiplicity. (This definition may be extended to the case where the coefficients of monomials are not non-negative integers, although we omit it here.) Since the monomials in $p_1 = \sum_i x_i$ are the variables $x_i$ themselves, $f[p_1] = f$, and so $p_1$ is the identity for the plethysm operation. Another example is
\[
p_{r_1}[p_{r_2}] = \sum_i (x_i^{r_2})^{r_1} = p_{r_1r_2},
\]
though in general plethysm is difficult to compute (it is an open problem to give a combinatorial intepretation of $\langle s_\lambda, s_{\mu}[s_\nu]\rangle$). Once $g$ is fixed, the mapping $f \to f[g]$ is an algebra homomorphism $\Lambda \to \Lambda$. If $f$ and $g$ are homogeneous, then $\deg(f[g]) = \deg(f) \deg(g)$.
\newline \newline \noindent
The multiset of monomials in $g + h$ is the union of the monomials in $g$ and those in $h$, which means that $f[g+h]$ may be computed by taking the comultiplication $\Delta(f) = \sum f^{(0)} \otimes f^{(1)}$ (where we use Sweedler notation), and then taking the plethysm with $g$ and $h$ in the separate components and multiplying:
\[ 
f[g+h] = \sum f^{(0)}[g] f^{(1)}[h].
\]
This extends directly to larger sums by taking successive comultiplications of $f$. In particular, if $g$ is not homogeneous, we may decompose $g = g_0 + g_1 + \cdots + g_d$ where $\deg(g_i) = i$ and $d$ is the degree of $g$. Then to compute $f[g]$, we may apply the comultiplication $\Delta$ $d$ times to $f$ (i.e. take $\Delta^d(f) = (\Delta \otimes \Id^{\otimes d}) \circ \cdots \circ (\Delta \otimes \Id) \circ \Delta$) to get an element of $\Lambda^{\otimes (d+1)}$, and the evaluate the $i$-th tensor factor at the monomials of $g_i$. Explicitly, let us write
\[
\Delta^{d}(f) = \sum f^{(0)} \otimes f^{(1)} \otimes \cdots \otimes f^{(d)},
\]
where Sweedler notation lets us assume that within each summand each $f^{(i)}$ is homogeneous (different summands may have different degrees of $f^{(i)}$). Then we conclude
\[
f[g] = \sum f^{(0)}[g_0] f^{(1)}[g_1] \cdots f^{(d)}[g_d].
\]
What will be essential for us is the operation adjoint to plethysm with $g$, which we denote $f[g^\perp]$.
\begin{proposition} \label{plethysm_formula_proposition}
Let $g = g_0 + g_1 + \cdots + g_d$. Then the adjoint operation to plethysm with $g$ satisfies
\[
f[g^\perp] = \sum f^{(0)}[g_0^\perp] f^{(1)}[g_1^\perp] \cdots f^{(d)}[g_d^\perp],
\]
where as before, $\Delta^{d}(f) = \sum f^{(0)} \otimes f^{(1)} \otimes \cdots \otimes f^{(d)}$ in Sweedler notation.
\end{proposition}
\begin{proof}
For any $h \in \Lambda$,
\begin{eqnarray*}
\langle f[g^\perp], h \rangle &=& \langle f, h[g] \rangle \\
&=& \langle f, \sum h^{(0)}[g_0] h^{(1)}[g_1] \cdots h^{(d)}[g_d] \rangle \\
&=&\langle \Delta^d(f), \sum h^{(0)}[g_0] \otimes  h^{(1)}[g_1] \otimes \cdots \otimes h^{(d)}[g_d] \rangle   \hspace{10mm} \mbox{(using the Hopf pairing property) } \\
&=&\langle \sum f^{(0)} \otimes f^{(1)} \otimes \cdots \otimes f^{(d)}, \sum h^{(0)}[g_0] \otimes  h^{(1)}[g_1] \otimes \cdots \otimes h^{(d)}[g_d] \rangle\\
&=&\langle \sum f^{(0)}[g_0^\perp] \otimes f^{(1)}[g_1^\perp] \otimes \cdots \otimes f^{(d)}[g_d^\perp], \sum h^{(0)} \otimes  h^{(1)} \otimes \cdots \otimes h^{(d)} \rangle \\
&=&\langle \sum f^{(0)}[g_0^\perp] \otimes f^{(1)}[g_1^\perp] \otimes \cdots \otimes f^{(d)}[g_d^\perp], \sum \Delta^d(h) \rangle \\
&=&\langle \sum f^{(0)}[g_0^\perp] f^{(1)}[g_1^\perp] \cdots f^{(d)}[g_d^\perp], h \rangle    \hspace{20mm} \mbox{(using the Hopf pairing property again) }.
\end{eqnarray*}
The proposition now follows from the nondegeneracy of the Hall inner product.
\end{proof}
\noindent
Let us restrict ourselves to the case where $g_i$ is homogeneous of degree $i$ and $g_0 = 0$. Then $f^{(0)}[g_0^\perp]$ is zero if $f^{(0)}$ is homogeneous of positive degree, and equal to $f^{(0)}$ if $\deg(f^{(0)}) = 0$. Since $g_i$ has degree $i$, $f_i[g_i^\perp]$ is zero unless $\deg(f_i)$ is a multiple of $i$. So the degree of $f^{(0)}[g_0^\perp] f^{(1)}[g_1^\perp] \cdots f^{(d)}[g_d^\perp]$ is
\[
\frac{\deg(f^{(1)})}{\deg(g_1)}+ \frac{\deg(f^{(2)})}{\deg(g_2)}+ \cdots + \frac{\deg(f^{(d)})}{\deg(g_d)} = \frac{\deg(f^{(1)})}{1}+ \frac{\deg(f^{(2)})}{2}+ \cdots + \frac{\deg(f^{(d)})}{d}.
\]
If $g_i$ is zero, then plethysm with $g_i$ annihilates positive degree symmetric functions. So in this case $f^{(i)}[g_i^\perp]$ is zero unless $\deg(f^{(i)}) = 0$, meaning we may omit $\frac{\deg(f^{(i)})}{i}$ from the above sum whenever $g_i = 0$.
\begin{observation} \label{skew_plethysm_degree_observation}
If $g_0 = g_1 = 0$, then $f[g^\perp]$ is a sum of terms of degree
\[
\frac{\deg(f^{(2)})}{2} + \cdots + \frac{\deg(f^{(d)})}{d} \leq \frac{\deg(f^{(2)})}{2} + \cdots + \frac{\deg(f^{(d)})}{2} \leq \frac{\deg(f)}{2},
\]
and hence $f[g^\perp]$ is (possibly inhomogeneous) of degree at most $\deg(f)/2$.
\end{observation}

\begin{remark}
Continue to assume that $f^{(0)}$ is homogeneous. In the case where $g_0$ is a nonzero constant $m$,
\begin{eqnarray*}
f^{(0)}[g_0^\perp] &=& \sum_\lambda \langle f^{(0)}[g_0^\perp], s_\lambda \rangle s_\lambda \\
&=& \sum_\lambda \langle f^{(0)}, s_\lambda(\underbrace{1,1, \ldots, 1}_{m\rm\ times}) \rangle s_\lambda,
\end{eqnarray*}
and the inner product is zero unless $\deg(f^{(0)}) = 0$. If $f^{(0)}$ is a constant, this sum becomes
\[
\sum_{\lambda} f^{(0)} s_\lambda(\underbrace{1,1, \ldots, 1}_{m\rm\ times}) s_{\lambda} = f^{(0)} \prod_{i \geq 1} \prod_{j=1}^m \frac{1}{1-x_i} = f^{(0)} H^m
\]
where we used the \emph{Cauchy identity}
\[
\sum_\lambda s_{\lambda}(x_i) s_{\lambda}(y_j) = \prod_{i,j} \frac{1}{1-x_iy_j}.
\]
In Equation \ref{restriction_multiplicity_equation}, the multiplicity of $S^\mu$ in the restriction of $\mathbb{S}^\lambda(\mathbb{C}^n)$ is given by
$\langle s_\lambda, s_\mu[H] \rangle$. This means that the restriction of $\mathbb{S}^\lambda(\mathbb{C}^n)$ to the symmetric group $S_n$ is determined by how the degree $n$ component of $s_\lambda[H^\perp]$ decomposes in terms of Schur functions $s_\mu$. Since $H$ has nonzero constant term, applying Proposition \ref{plethysm_formula_proposition} requires this calculation.
\end{remark}
\begin{definition}
The $n$-th \emph{Lyndon symmetric function} is 
\[
L_n = \frac{1}{n} \sum_{d | n} \mu(n/d) p_d,
\]
where $\mu$ is the M\"{o}bius function. The \emph{total Lyndon symmetric function}
\[
L = \sum_{n \geq 1} L_n,
\]
which is an element of the completion of $\Lambda$ with respect to the grading.
\end{definition}
\noindent
One way that $L$ arises is as the $GL(V)$ character of the free Lie algebra on $V$ (\cite{Macdonald} Section 1.7, Example 12). We observe that $L$ has no constant term, and we may write it as $L = p_1 + L_{\geq 2}$, where $L_{\geq 2} = \sum_{n \geq 2} L_n$ only contains terms of degrees 2 and greater. Hence Observation \ref{skew_plethysm_degree_observation} gives us the following.
\begin{observation} \label{lyndon_specific_observation}
We have the formula
\[
f[L^\perp] = \sum f^{(0)}[p_1^\perp] \cdot f^{(1)}[L_{\geq 2}^\perp] = \sum f^{(0)} \cdot (f^{(1)}[L_{\geq 2}^\perp]),
\]
where $f^{(1)}[L_{\geq 2}^\perp]$ is a (possibly inhomogeneous) symmetric function of degree at most $\deg(f^{(1)})/2$.
\end{observation}

\section{Stable Characters}
\noindent
In \cite{OrellanaZabrocki}, the \emph{irreducible character basis} $\tilde{s}_\lambda$ of $\Lambda$ was defined. These symmetric functions were independently studied in \cite{AssafSpeyer} under the name \emph{stable Specht polynomials} and notation $s_\lambda^\dagger$. In the present paper we use the notation $\tilde{s}_\lambda$ because we continue the study of these functions under the Hopf algebra structure of $\Lambda$ initiated in \cite{OrellanaZabrocki2}. The $\tilde{s}_\lambda$ also appear in Section 10 of \cite{Ryba1} and implicitly in Section 8 of \cite{Tosteson}. The author also provided a categorification of $\tilde{s}_\lambda$ in \cite{Ryba2}, connected to the category $\mathrm{Rep}(S_\infty)$ studied by Sam and Snowden \cite{SamSnowden2, SamSnowden1}.
\newline \newline \noindent
Recall that we view $S_n$ as a subgroup of $GL_n(\mathbb{C})$, so we may speak of the eigenvalues of permutations, and also that $\lambda[n]$ denotes $(n-|\lambda|, \lambda)$, which is a partition for $n$ large enough.
\begin{proposition}
There is a $\mathbb{Z}$-module basis of $\Lambda$, $\tilde{s}_\lambda$, indexed by all partitions $\lambda$. These symmetric functions obey the following properties:
\begin{itemize}
\item[1.] $\tilde{s}_\lambda$ is inhomogeneous of degree $|\lambda|$ and the leading order term is the Schur function $s_\lambda$.
\item[2.] If $n \geq 2|\lambda|$, then evaluating $\tilde{s}_\lambda$ at the eigenvalues of $g \in S_n$ gives the character value $\chi^{\lambda[n]}(g)$.
\item[3.] We have the formula
\[
\tilde{s}_\lambda = \sum_{\mu} (-1)^{|\lambda|-|\mu|} \langle s_{\lambda^\prime}, s_{\mu^\prime}[L]H \rangle s_\mu.
\]
\end{itemize}
\end{proposition}
\begin{proof}
Part 1 is follows from by taking the leading order term in Theorem 14 of \cite{OrellanaZabrocki} or Theorem 1.2 of \cite{AssafSpeyer}. Part 2 is (part of) Theorem 1 of \cite{OrellanaZabrocki}, although it is convenient for us to use the weaker bound of $n \geq 2|\lambda|$ rather than the given $n \geq |\lambda| + \lambda_1$. Part 3 is Theorem 1.2 of \cite{AssafSpeyer}, but rewritten using the Pieri rule: the factor of $H$ replaces a sum over vertical/horizontal strips.
\end{proof}
\noindent
By the first part of the proposition, if we wish to express a symmetric function $f$ in the basis $\tilde{s}_\lambda$, we only need to use partitions satisfying $|\lambda| \leq \deg(f)$.

\begin{remark}
In the case where $n$ is too small for $\lambda[n]$ to be a partition, the behaviour of evaluations of $\tilde{s}_\lambda$ was established in \cite{Ryba2}. The result is either zero or (up to a sign) the character of an irreducible representation of $S_n$ determined by turning $(n-|\lambda|, \lambda)$ into a partition using a certain algorithm.
\end{remark}
\noindent
The formula for $\tilde{s}_\lambda$ may be expressed in terms of differential operators applied to the Schur function $s_\lambda$.
\begin{proposition} \label{diff_op_proposition}
There exist differential operators $D_i$, homogeneous of degree $-i$, such that
\[
\tilde{s}_\lambda = \sum_{i \geq 0} D_i(s_\lambda).
\]
Moreover each $D_i$ may be written as a linear combination of $p_\alpha \frac{\partial}{\partial p_\beta}$ where $|\beta| \leq 2i$ (and by degree considerations $|\alpha|-|\beta| = -i$).
\end{proposition}
\begin{proof}
We apply the isometry $S$ to $\tilde{s}_\lambda$ to get
\begin{eqnarray*}
S(\tilde{s}_\lambda) &=& \sum_\mu (-1)^{|\lambda|} \langle s_{\lambda^\prime}, s_{\mu^\prime}[L]H \rangle s_{\mu^\prime} \\
&=& \sum_\mu (-1)^{|\lambda|} \langle (H^\perp s_{\lambda^\prime})[L^\perp], s_{\mu^\prime} \rangle s_{\mu^\prime} \\
&=& (-1)^{|\lambda|} (H^{\perp} s_{\lambda^{\prime}})[L^\perp].
\end{eqnarray*}
To proceed, we note that
\begin{eqnarray*}
\langle \Delta( H^\perp f), g \otimes h \rangle &=& \langle H^\perp f, gh \rangle \\
&=& \langle f, Hgh \rangle \\
&=& \langle \Delta(f), g \otimes Hh \rangle \\
&=& \langle (\Id \otimes H^\perp) \Delta(f), g \otimes h \rangle,
\end{eqnarray*}
so $\Delta(H^\perp f) = (\Id \otimes H^\perp) \Delta(f)$. Taking $f = s_{\lambda^\prime}$ and applying Lemma \ref{standard_coproduct_formula_lemma}, we have
\[
\Delta(H^\perp s_{\lambda^\prime}) = (\Id \otimes H^\perp) \Delta(s_{\lambda^\prime}) = \sum_{\mu} p_\mu^\perp s_{\lambda^\prime} \otimes H^\perp \frac{p_\mu}{z_\mu}.
\]
Next we apply Observation \ref{lyndon_specific_observation}:
\[
(H^{\perp} s_{\lambda^{\prime}})[L^\perp] = \sum_{\nu} (p_\nu^\perp s_{\lambda^\prime}) \cdot ((H^\perp \frac{p_\nu}{z_\nu})[L_{\geq 2}^\perp]).
\]
We let $j_\nu = ((H^\perp \frac{p_\nu}{z_\nu})[L_{\geq 2}^\perp])$ and note that the degree of $j_\nu$ is at most $|\nu|/2$ because $H^\perp p_\nu$ has degree $|\nu|$ (although it is inhomogeneous). The degree of $p_\nu^\perp s_{\lambda^\prime}$ is $|\lambda|-|\nu|$. To recover $\tilde{s}_\lambda$ we apply $S$ again and use Lemma \ref{antipode_adjoint_lemma}:
\begin{eqnarray*}
\tilde{s}_\lambda &=& (-1)^{|\lambda|}\sum_{\nu} S(j_\nu) S(p_\nu^\perp s_{\lambda^\prime}) \\
&=& \sum_{\nu} S(j_\nu) (-1)^{l(\nu)} p_\nu^\perp s_{\lambda}.
\end{eqnarray*}
We have now expressed $\tilde{s}_\lambda$ as a sum of differential operators applied to $s_\lambda$. The degree of (the homogeneous components of) the operator $S(j_\nu) p_\nu^\perp$ is at most $\frac{|\nu|}{2} - |\nu| = \frac{-|\nu|}{2}$, so only $\nu$ with $2i \geq |\nu|$ contribute to $D_i$.
\end{proof}

\section{Kronecker Comultiplication of Stable Characters}
\noindent
Recall that $[k]$ denotes $\{1,2,\ldots k\}$. If $m,n \in \mathbb{Z}_{\geq 0}$, the groups $S_m$ and $S_n$ act on the two factors of $[m] \times [n]$. Since this set has $mn$ elements, we obtain a homomorphism $\varphi: S_{m} \times S_{n} \to S_{mn}$. Linearising this action of $S_m \times S_n$ on $[m]\times[n]$ gives a representation of $S_m \times S_n$ on $\mathbb{C}^m \otimes \mathbb{C}^n = \mathbb{C}^{mn}$ which we now describe. For $(i,j) \in [m] \times [n]$, $\mathbb{C}^m \otimes \mathbb{C}^n$ has basis $e_i \otimes e_j$. Then for $(g_1, g_2) \in S_m \times S_n$,
\[
(g_1, g_2) \cdot (e_i \otimes e_j) = e_{g_1(i)} \otimes e_{g_2(j)}.
\]
Thus the permutation matrix associated to $(g_1, g_2)$ (viewed as an element of $S_{mn}$) is the tensor product of the permutation matrices associated to $g_1$ and $g_2$. This allows us to describe the eigenvalues of $\varphi(g_1, g_2)$.
\begin{observation} \label{tensor_product_eigenvalue_observation}
Suppose that the eigenvalues of $g_1 \in S_m$ are $\alpha_i$ ($i \in [m]$) and the eigenvalues of $g_2 \in S_n$ are $\beta_j$ ($j \in [n]$). Then the eigenvalues of $\varphi(g_1, g_2)$ are the products $\alpha_i \beta_j$ ($(i,j) \in [m] \times [n]$).
\end{observation}
\noindent
This in turn allows us to give an interpretation of the Kronecker comultiplication of stable characters.
\begin{theorem} \label{restriction_coefficient_theorem}
Let the integers $R_{\mu, \nu}^\lambda$ be defined by
\[
\Delta^*(\tilde{s}_\lambda) = \sum_{\mu, \nu} R_{\mu, \nu}^\lambda \tilde{s}_\mu \otimes \tilde{s}_\mu.
\]
Then $R_{\mu, \nu}^\lambda$ is the multiplicity of $S^{\mu[m]} \otimes S^{\nu[n]}$ in
\[
\mathrm{Res}_{S_m \times S_n}^{S_{mn}}(S^{\lambda[mn]}),
\]
provided that $m$ and $n$ are each sufficiently large relative to $\lambda$ ($m,n \geq 2|\lambda|$ suffices).
\end{theorem}
\begin{proof}
To compute the character of $S^{\lambda[mn]}$ on the element $g$, we evaluate $\tilde{s}_\lambda$ at the eigenvalues of $g$. Suppose now that $g_1 \in S_m$ has eigenvalues $\alpha_i$, $g_2 \in S_m$ has eigenvalues $\beta_j$, and $g = \varphi(g_1,g_2)$. By Observation \ref{tensor_product_eigenvalue_observation},
\begin{eqnarray*}
\chi^{\lambda[mn]}(g) &=& \tilde{s}_\lambda(\alpha_i\beta_j) \\
&=& \sum_{\mu, \nu} R_{\mu, \nu}^\lambda \tilde{s}_{\mu}(\alpha_i) \otimes \tilde{s}_\nu(\beta_j) \\
&=& \sum_{\mu, \nu} R_{\mu,\nu}^\lambda \chi^{\mu[m]}(g_1) \cdot \chi^{\nu[n]}(g_2).
\end{eqnarray*}
Since $\Delta^*$ takes degree $d$ symmetric functions to (sums of) tensor products of two degree $d$ symmetric functions, only $\mu, \nu$ with $|\mu|,|\nu| \leq |\lambda|$ arise, so $m,n \geq 2|\lambda|$ guarantees that $\mu[m]$ and $\nu[n]$ are always valid partitions. Viewing this equation as an equality of characters of $S_m \times S_n$, we obtain the theorem.
\end{proof}

\begin{remark} \label{deligne_category_remark}
Theorem \ref{restriction_coefficient_theorem} in particular shows that the restriction multiplicities from $S_{mn}$ to $S_m \times S_n$ obey a certain stability property (they are constant for $m$ and $n$ sufficiently large). There is another way to see this via the Deligne category $\underline{Rep}(S_t)$ (we direct the reader to \cite{ComesOstrik} for an introduction). The category is constructed from an object $V$ which ``interpolates'' the permutation representations $\mathbb{C}^n$ of the symmetric groups $S_n$; when $t$ is a non-negative integer, there are ``specialisation'' functors from $\underline{Rep}(S_t)$ to the usual representation category of the symmetric group $S_t$. It turns out that $V$ has the structure of a commutative Frobenius algebra of categorical dimension $t$. By Proposition 8.3 of \cite{Deligne}, if $\mathcal{C}$ is $\mathbb{C}$-linear symmetric tensor category, symmetric monoidal functors $F:\underline{Rep}(S_t) \to \mathcal{C}$ are equivalent to commutative Frobenius algebras in $\mathcal{C}$ of categorical dimension $t$, where a functor $F$ corresponds to the object $F(V)$. One example of such a functor is
\[
F: \underline{Rep}(S_{t_1t_2}) \to \underline{Rep}(S_{t_1}) \boxtimes \underline{Rep}(S_{t_2}),
\]
where $F(V) = V \otimes V$ is the product of the generating objects of the two copies of the Deligne category. This is the ``interpolated'' version of $\Res{S_m \times S_n}{S_{mn}}(\mathbb{C}^{mn}) = \mathbb{C}^m \otimes \mathbb{C}^n$. Stability properties follow by taking $t_1, t_2$ to be sufficiently large integers and applying specialisation functors to obtain representations of $S_{t_1} \times S_{t_2}$. The argument is similar to stable Kronecker coefficients in Subsection 5.4 of \cite{Entova-Aizenbud}. The author is grateful to Pavel Etingof for this remark.
\end{remark}

\section{Two-Row Stability}
\noindent
Theorem \ref{restriction_coefficient_theorem} shows that the limit of the restriction multiplicities
\[
\lim_{m,n \to \infty}
\dim\hom_{S_m \times S_n}(S^{\mu[m]} \boxtimes S^{\nu[n]}, \Res{S_m\times S_n}{S_{mn}}(S^{\lambda[mn]}))
\]
exists and is finite. The aim of the rest of the paper is to extend this to a similar stability result where two rows of all the partitions are variable. We achieve this via a computation of $\Delta^*(\tilde{s}_{\lambda[n]})$ and determining the behaviour for $n$ large. Our approach is to understand the transition matrices between the $\tilde{s}_\lambda$ and Schur function bases.

\begin{proposition} \label{first_stable_decomposition_proposition}
Let $i$ be an integer. Suppose that for partitions $\lambda, \mu$, the sequence $M_{\lambda, \mu}^{(i)}(n)$ is defined via
\[
\tilde{s}_{\lambda[n]} = \sum_{i,\mu} M_{\lambda, \mu}^{(i)}(n) s_{\mu[n-i]},
\]
where we take $M_{\lambda, \mu}^{(i)}(n) = 0$ if either $\mu[n-i]$ or $\lambda[n]$ is not defined. Then $M_{\lambda, \mu}^{(i)}(n)$ is constant for $n \geq 2 |\lambda| + 3i$. Let $
M_{\lambda, \mu}^{(i)}$ be the limiting value. If $i < 0$, then $M_{\lambda, \mu}^{(i)} = 0$. If $i=0$, then $M_{\lambda, \mu}^{(0)} = \delta_{\lambda, \mu}$. If $i \geq 0$, then $M_{\lambda, \mu}^{(i)}(n) = 0$ unless $||\mu|-|\lambda|| \leq 2i$.
\end{proposition}

\begin{proof}
Since $\tilde{s}_{\lambda[n]}$ is of degree $n$, it cannot contain any Schur functions of degree larger than $n$, so the case $n-i > n$ is ruled out. We use the characterisation of $\tilde{s}_\lambda$ from Proposition \ref{diff_op_proposition}. We write
\[
\tilde{s}_{\lambda[n]} = \sum_{i \geq 0} D_i s_{\lambda[n]},
\]
and express each $D_i$ as a linear combination of $p_\alpha \frac{\partial}{\partial p_\beta}$ (where $|\beta| - |\alpha| = i$ and $|\beta| \leq 2i$). By comparing degrees, we see that
\[
\sum_\mu M_{\lambda, \mu}^{(i)} s_{\mu[n-i]} = D_i s_{\lambda[n]}.
\]
When $i=0$, $D_i$ is the identity, and $D_0 s_{\lambda[n]} = s_{\lambda[n]}$, proving the statement about $M_{\lambda, \mu}^{(0)}$. Finally, applying $p_\alpha \frac{\partial}{\partial p_\beta}$ amounts to some sequence of removing and then adding rim hooks to $\lambda[n]$. By Observation \ref{rim_hook_observation}, if $n \geq 2|\lambda| + r$, then any rim hook of $\lambda[n]$ of size at most $r$ is either contained in the top row, or contained in $\lambda$. Applying $\frac{\partial}{\partial p_\beta}$ amounts to successively removing several rim hooks of total size $|\beta|$. As soon as $n \geq 2|\lambda| + |\beta|$, Observation \ref{rim_hook_observation} will apply at each step of rim hook removal, and each rim hook will be removed either from the top row or from $\lambda$. Then applying $p_\alpha$ amounts to adding some rim hooks. But if the top row is at least $|\alpha|$ longer than the second row, these rim hooks are either added entirely to the top row, or entirely to $\lambda$. So as soon as $n \geq 2|\lambda| + |\beta| + |\alpha|$, the result of the rim hook removal and addition ceases to change as $n$ gets larger. Since $|\beta| \leq 2i$ and $|\alpha| = -i+|\beta| \leq i$, we have $|\beta| + |\alpha| \leq 3i$. So $M_{\lambda, \mu}^{(i)}(n)$ is constant for $n \geq 2|\lambda| + 3i$. Finally, at most $2i$ boxes are removed from $\lambda$ and then at most $i$ boxes are added. In any case, the resulting partition $\mu$ has size within $2i$ of the size $|\lambda|$.
\end{proof}
\noindent
As mentioned in the introduction, the change-of-basis matrix between the Schur functions $s_\mu$ and the $\tilde{s}_\lambda$ encodes the multiplicities of the irreducible representations $S^{\lambda[n]}$ in $\mathbb{S}^{\mu}(\mathbb{C}^n)$ provided $n$ is large enough. The following proposition implies stability for the multiplicity of $S^{\lambda[n-i][m]}$ in $\mathbb{S}^{\mu[n]}(\mathbb{C}^m)$, where we first take $m \to \infty$ and then $n \to \infty$ (here $\lambda[n-i][m] = (n-i-|\lambda|, \lambda)[m] = (m-n+i, n-i-|\lambda|, \lambda)$).

\begin{proposition}
For partitions $\lambda, \mu$, define the sequence $N_{\mu, \lambda}^{(i)}(n)$ via
\[
s_{\mu[n]} = \sum_{i,\lambda} N_{\mu, \lambda}^{(i)}(n) \tilde{s}_{\lambda[n-i]},
\]
where we take $N_{\mu, \lambda}^{(i)}(n) = 0$ if either $\mu[n]$ or $\lambda[n-i]$ is undefined. Then $N_{\mu, \lambda}^{(i)}(n)$ is constant for $n \geq 2 |\lambda| + 7i$. Let $N_{\mu,\lambda}^{(i)}$ be the limiting value. If $i < 0$, $N_{\mu, \lambda}^{(i)} = 0$. Also, $N_{\mu, \lambda}^{(0)} = \delta_{\mu, \lambda}$, and $N_{\mu, \lambda}^{(i)} = 0$ unless $||\mu| - |\lambda|| \leq 2i$.
\end{proposition}
\begin{proof}
First of all, degree considerations imply that when we express $s_{\mu[n]}$ (which has degree $n$), only $\tilde{s}_{\lambda[n-i]}$ (which has degree $n-i$) with $n-i \leq n$ can appear. Thus if $i < 0$, $N_{\mu, \lambda}^{(i)} = 0$. 
\newline \newline \noindent
We use the formula from Proposition \ref{first_stable_decomposition_proposition} for $\tilde{s}_{\lambda[n-i]}$ (we substitute $n \to n-i$) to get
\[
s_{\mu[n]} = \sum_{i,\lambda} N_{\mu, \lambda}^{(i)}(n) \sum_{j,\nu} M_{\lambda, \nu}^{(j)}(n-i) s_{\nu[n-i-j]}.
\]
Equating coefficients of Schur functions we see that for $n$ large enough that $\mu[n]$ is defined,
\[
\sum_{\lambda} \sum_{i+j=k} N_{\mu, \lambda}^{(i)}(n) M_{\lambda, \nu}^{(j)}(n-i) = \delta_{\mu, \nu}\delta_{k,0}.
\]
We may assume that only $\lambda$ with $||\lambda| - |\nu|| \leq 2j$ appear in the sum, otherwise $M_{\lambda, \nu}^{(j)}(n-i) = 0$. In particular, $|\lambda| \leq |\nu| + 2k$, so only finitely many $|\lambda|$ appear. When $n-i \geq 2|\lambda| +3j$, $M_{\lambda, \nu}^{(j)}(n-i)$ takes its limiting value $M_{\lambda, \nu}^{(j)}$ (which is independent of $i$). So since
\[
2|\nu| + 7k \geq 2 |\lambda| + 3k \geq 2 |\lambda| + 3j +i,
\]
taking $n \geq 2|\nu| + 7k$ eliminates the $n$-dependence of all the $M_{\lambda, \nu}^{(j)}(n-i)$ in the expression. Since $M_{\lambda, \nu}^{(0)} = \delta_{\lambda, \nu}$ the resulting equations are
\[
N_{\mu, \nu}^{(k)}(n) + 
\sum_{\lambda} \sum_{i=0}^{k-1} N_{\mu, \lambda}^{(i)}(n) M_{\lambda, \nu}^{(k-i)} = \delta_{\mu, \nu}\delta_{k,0}.
\]
These equations uniquely determine $N_{\mu, \lambda}^{(k)}(n)$ by induction on $k$. Moreover, since the coefficients $M_{\lambda, \nu}^{(k-i)}$ of these linear equations are constant once $n \geq 2|\nu| + 7k$, the same is true of the solutions $N_{\mu, \nu}^{(k)}(n)$.
\newline \newline \noindent
Taking $n$ sufficiently large, we now have the equations
\[
\sum_{\lambda} \sum_{i+j=k} N_{\mu, \lambda}^{(i)} M_{\lambda, \nu}^{(j)} = \delta_{\mu, \nu}\delta_{k,0}.
\]
We rewrite this as a matrix equation, where the matrices have entries that are formal power series in a variable $x$. Let
\[
\mathbf{M}_{\lambda, \nu} = \sum_{j \geq 0} M_{\lambda, \nu}^{(j)} x^j, \hspace{5mm}\mathbf{N}_{\mu, \lambda} = \sum_{i \geq 0} N_{\mu, \lambda}^{(j)} x^i,
\]
so that our relations now become $\mathbf{N}\mathbf{M} = \Id$. Since the degree zero component of $\mathbf{M}$ is the identity matrix, if we write $\mathbf{M} = \Id + \mathbf{M}_{>0}$, then $\mathbf{M}_{>0}$ has positive $x$-adic valuation. Hence, 
\[
\mathbf{N} = \mathbf{M}^{-1} = \sum_{k \geq 0} (-1)^k \mathbf{M}_{>0}^k,
\]
which converges in the $x$-adic topology. An entry $(\mathbf{M}_{>0}^k)_{\alpha_0, \alpha_{k}}$ is a linear combination of terms of the form
\[
M_{\alpha^{(0)}\alpha^{(1)}}^{(j_1)} x^{j_1}
M_{\alpha^{(1)}\alpha^{(2)}}^{(j_2)} x^{j_2}
\cdots
M_{\alpha^{(k-1)}\alpha^{(k)}}^{(j_{k})} x^{j_k},
\]
where each $j_r \geq 1$. This term has degree $i = j_1 + j_2 + \cdots + j_k$. Such a term is zero unless we have
\begin{eqnarray*}
||\alpha_1|-|\alpha_0|| &\leq&  2j_1 \\
||\alpha_2|-|\alpha_1|| &\leq&  2j_2 \\
&\cdots& \\
||\alpha_{k}|-|\alpha_{k-1}|| &\leq&  2j_k \\
\end{eqnarray*}
which implies $||\alpha_{k}| - |\alpha_0|| \leq 2(j_1 + \cdots + j_k) = 2i$. We conclude that $N_{\mu, \lambda}^{(i)}$ must be zero unless $||\lambda| - |\mu|| \leq 2i$.
\end{proof}

\begin{theorem} \label{R_stability_theorem}
Fix integers $a, b$. We have that
\[
\lim_{n \to \infty} R_{\mu[n-a], \nu[n-b]}^{\lambda[n]}
\]
exists and is finite.
\end{theorem}
\begin{proof}
Fix $i \geq 0$. By the definition of $M_{\lambda, \rho}^{(i)}(n)$, we have
\[
\Delta^*(\tilde{s}_{\lambda[n]}) = \sum_{i \geq 0} \sum_{\rho} M_{\lambda, \rho}^{(i)}(n) \Delta^*(s_{\rho[n-i]}).
\]
Since $\Delta^*$ acts by Kronecker coefficients in the Schur function basis,
\[
\sum_{\rho} M_{\lambda, \rho}^{(i)}(n) \Delta^*(s_{\rho[n-i]}) = \sum_{\rho} M_{\lambda, \rho}^{(i)}(n) \sum_{\sigma, \tau} k_{\sigma[n-i], \tau[n-i]}^{\rho[n-i]} s_{\sigma[n-i]} \otimes s_{\tau[n-i]},
\]
where we sum over $\rho, \sigma, \tau$ such that $\rho[n-i], \sigma[n-i], \tau[n-i]$ are defined. Now we use $N_{\mu, \lambda}^{(j)}(n)$ to change basis again and obtain an expression for $\Delta^*$ in the $\tilde{s}_{\lambda[n]}$ basis:
\[
\Delta^*(\tilde{s}_{\lambda[n]}) = \sum_{i \geq 0} \sum_{\rho} M_{\lambda, \rho}^{(i)}(n) \sum_{\sigma, \tau} k_{\sigma[n-i], \tau[n-i]}^{\rho[n-i]} \sum_{\mu, a \geq i} N_{\sigma, \mu}^{(a-i)}(n-i) \tilde{s}_{\mu[n-a]} \otimes \sum_{\nu, b \geq i} N_{\tau, \nu}^{(b-i)}(n-i) \tilde{s}_{\nu[n-b]}.
\]
In particular, we have
\[
R_{\mu[n-a], \nu[n-b]}^{\lambda[n]} =
\sum_{i \geq 0} \sum_{\rho} M_{\lambda, \rho}^{(i)}(n) \sum_{\sigma, \tau} k_{\sigma[n-i], \tau[n-i]}^{\rho[n-i]} \sum_{\mu} N_{\sigma, \mu}^{(a-i)}(n-i) \sum_{\nu} N_{\tau, \nu}^{(b-i)}(n-i).
\]
Now we observe that when $n$ is large, the only nonzero terms in this sum occur when $a-i \geq 0$, $b-i \geq 0$ $i \geq 0$, $||\sigma|- |\mu|| \leq 2(a-i)$, $||\tau|- |\nu|| \leq 2(b-i)$, and $||\rho|- |\lambda|| \leq 2i$. These conditions in particular imply
\begin{eqnarray*}
|\rho| &\leq& |\lambda| +2\min(a,b)\\
|\sigma| &\leq& |\mu| + 2a\\
|\tau| &\leq& |\nu| + 2b.
\end{eqnarray*}
Thus there are only finitely many $\rho, \sigma, \tau$ that are relevant to the sum. For each triple $\rho, \sigma, \tau$ we invoke the stability of Kronecker coefficients to get
\[
\lim_{n \to \infty} R_{\mu[n-a], \nu[n-b]}^{\lambda[n]} =
\sum_{i \leq \min(a,b)} \sum_{\rho} M_{\lambda, \rho}^{(i)} \sum_{\sigma, \tau} \tilde{k}_{\sigma, \tau}^{\rho} \sum_{\mu} N_{\sigma, \mu}^{(a-i)} \sum_{\nu} N_{\tau, \nu}^{(b-i)},
\]
to finally eliminate dependence on $n$, proving that the limit exists and is finite.
\end{proof}
\noindent
Finally, we rephrase this in terms of restriction multiplicities.
\begin{corollary}
For any $a,b$ we have that
\[
\lim_{n \to \infty} \lim_{p,q \to \infty} \dim\hom_{S_p \times S_q}\left(S^{\mu[n-a][p]} \otimes S^{\nu[n-b][q]}, \Res{S_p \times S_q}{S_{pq}} \left( S^{\lambda[n][pq]} \right) \right)
\]
exists and is finite, where $\rho[r][s] = (r - |\rho|, \rho)[s] = (s-r, r-|\rho|, \rho)$.
\end{corollary}
\noindent
It would be interesting to find a categorical interpretation of this two-row stability pattern similar to how the Deligne category explains the one-row stability pattern in Remark \ref{deligne_category_remark}.

\bibliographystyle{alpha}
\bibliography{ref.bib}

\begin{thebibliography}{OZ21b}

\bibitem[AS20]{AssafSpeyer}
Sami Assaf and David Speyer.
\newblock Specht modules decompose as alternating sums of restrictions of schur
  modules.
\newblock {\em Proceedings of the American Mathematical Society},
  148(3):1015--1029, 2020.

\bibitem[CO09]{ComesOstrik}
Jonathan Comes and Victor Ostrik.
\newblock On blocks of deligne's category rep (s\_t).
\newblock {\em arXiv preprint arXiv:0910.5695}, 2009.

\bibitem[Del07]{Deligne}
Pierre Deligne.
\newblock La cat{\'e}gorie des repr{\'e}sentations du groupe sym{\'e}trique st,
  lorsque t n’est pas un entier naturel.
\newblock {\em Algebraic groups and homogeneous spaces}, 19:209--273, 2007.

\bibitem[EA16]{Entova-Aizenbud}
Inna Entova-Aizenbud.
\newblock Deligne categories and reduced kronecker coefficients.
\newblock {\em Journal of Algebraic Combinatorics}, 44(2):345--362, 2016.

\bibitem[Gay76]{Gay}
David~A. Gay.
\newblock Characters of weyl group of $su(n)$ on zero weight spaces and
  centralizers of permutation representations.
\newblock {\em Rocky Mountain J. Math.}, 6(3):449--456, 09 1976.

\bibitem[GR20]{GrinbergReiner}
Darij Grinberg and Victor Reiner.
\newblock Hopf algebras in combinatorics, 2020.

\bibitem[Mac95]{Macdonald}
I.G. Macdonald.
\newblock {\em Symmetric Functions and Hall Polynomials}.
\newblock Oxford Classic Texts in the Physical Sciences. Clarendon Press, 1995.

\bibitem[MM08]{MaiaMendez}
Manuel Maia and Miguel M{\'e}ndez.
\newblock On the arithmetic product of combinatorial species.
\newblock {\em Discrete Mathematics}, 308(23):5407--5427, 2008.

\bibitem[Mur38]{Murnaghan}
Francis~D Murnaghan.
\newblock The analysis of the kronecker product of irreducible representations
  of the symmetric group.
\newblock {\em American journal of mathematics}, 60(3):761--784, 1938.

\bibitem[OZ21a]{OrellanaZabrocki2}
Rosa Orellana and Mike Zabrocki.
\newblock The {Hopf} structure of symmetric group characters as symmetric
  functions.
\newblock {\em Algebraic Combinatorics}, 4(3):551--574, 2021.

\bibitem[OZ21b]{OrellanaZabrocki}
Rosa Orellana and Mike Zabrocki.
\newblock Symmetric group characters as symmetric functions.
\newblock {\em Advances in Mathematics}, 390:107943, 2021.

\bibitem[Ryb19]{Ryba1}
Christopher Ryba.
\newblock Stable grothendieck rings of wreath product categories.
\newblock {\em Journal of Algebraic Combinatorics}, 49(3):267--307, 2019.

\bibitem[Ryb21]{Ryba2}
Christopher Ryba.
\newblock Littlewood complexes for symmetric groups.
\newblock {\em Representation Theory of the American Mathematical Society},
  25(20):594--605, 2021.

\bibitem[SF97]{Stanley}
R.P. Stanley and S.~Fomin.
\newblock {\em Enumerative Combinatorics, Volume 2}.
\newblock Cambridge Studies in Advanced Mathematics. Cambridge University
  Press, 1997.

\bibitem[SS15a]{SamSnowden3}
Steven~V Sam and Andrew Snowden.
\newblock Stability patterns in representation theory.
\newblock In {\em Forum of Mathematics, Sigma}, volume~3. Cambridge University
  Press, 2015.

\bibitem[SS15b]{SamSnowden2}
Steven~V Sam and Andrew Snowden.
\newblock Stability patterns in representation theory.
\newblock In {\em Forum of Mathematics, Sigma}, volume~3. Cambridge University
  Press, 2015.

\bibitem[SS16]{SamSnowden1}
Steven Sam and Andrew Snowden.
\newblock Gl-equivariant modules over polynomial rings in infinitely many
  variables.
\newblock {\em Transactions of the American Mathematical Society},
  368(2):1097--1158, 2016.

\bibitem[Tos21]{Tosteson}
Philip Tosteson.
\newblock Categorifications of rational hilbert series and characters of
  $fs^{op}$ modules.
\newblock {\em arXiv preprint arXiv:2105.01711}, 2021.

\end{thebibliography}

\end{document}